\documentclass{article}
\usepackage[dvipdfmx]{graphicx}

\usepackage[fleqn]{amsmath}
\usepackage{amsthm}
\usepackage{amsmath}
\usepackage{amssymb}
\usepackage{amsbsy}
\usepackage{amsfonts}
\usepackage{mathrsfs}
\usepackage[all]{xy}
\usepackage{amstext}
\usepackage{amscd}
\usepackage[dvips]{epsfig}
\usepackage{psfrag}
\usepackage{enumerate}
\usepackage{flafter}

\usepackage{mathtools}
\usepackage{esint}
\allowdisplaybreaks

\theoremstyle{plain}
\newtheorem{thm}{Theorem}[section]
\newtheorem{prop}[thm]{Proposition}
\newtheorem{cor}[thm]{Corollary}
\newtheorem{lem}[thm]{Lemma}
\theoremstyle{definition}
\newtheorem{exa}[thm]{Example}

\newtheorem{defn}[thm]{Definition}
\newtheorem{prob}[thm]{Problem}

\def\Ker{\mathop{\mathrm{Ker}}\nolimits}

\def\Hom{\mathop{\mathrm{Hom}}\nolimits}

\newcommand{\lra}{\longrightarrow}
\newcommand{\ra}{\rightarrow}
\newcommand{\Q}{{\Bbb Q}}

\newcommand{\Z}{{\Bbb Z}}
\newcommand{\N}{{\Bbb N}}

\newcommand{\K}{{\mathbb{K}}}

\newcommand{\aug}{{\mathrm{aug}}}

\newcommand{\sh}{\mathcal{S}}

\begin{document}

\large
\begin{center}
{\bf\Large Fox pairings of Poincar\'{e} duality groups}
\end{center}
\vskip 1.5pc
\begin{center}{Takefumi Nosaka\footnote{
E-mail address: {\tt nosaka@math.titech.ac.jp}
}}\end{center}
\vskip 1pc
\begin{abstract}\baselineskip=12pt \noindent
This paper develops the study of Fox pairings of a group $G$ from the viewpoint of group cohomology. We compute some cohomology groups of Fox pairings of $G$, where $G$ admits a Poincar\'{e} duality group pair. We also suggest fundamental Fox pairings and higher Fox pairings.
\end{abstract}
\begin{center}
\normalsize
\baselineskip=17pt
{\bf Keywords} \\
\ Fox pairing, group cohomology, Poincar\'{e} duality, derivations 
\end{center}

\tableofcontents

\large
\baselineskip=16pt
\section{Introduction}
\label{IntroS}
Let $\K[G] $ be the group ring of a group $G$ over a commutative ring $\K$. Let $M$ be a $\K[G]$-bimodule and $\aug: \K[G] \ra \K$ be the augmentation map. A {\it Fox pairing (of $G$)} \cite{MT,MT2,Tur,Tur2} is defined to be a $\K$-bilinear map $\eta: \K[G] \times \K[G] \ra M$ satisfying
\begin{equation}\label{t1} \eta(a_1 a_2,b)= \eta(a_1 ,b) \mathrm{aug}(a_2)+ a_1 \eta( a_2,b), \ \ \mathrm{for \ any \ } a_1,a_2,b \in \K[G], \end{equation}
\begin{equation}\label{t2}\eta(a,b_1b_2)= \eta(a ,b_1) b_2 + \mathrm{aug}(b_1) \eta( a,b_2), \ \ \mathrm{for \ any \ } a,b_1,b_2 \in \K[G] .\end{equation}
See \cite[Sections 6--8]{MT2}, Example \ref{deriex}, and Appendix \ref{SS2} for examples. There have been studies of Fox pairings and their applications in the case that $G$ is a surface group and $M=\K[G]$; for example, they give a generalization of (logarithms of) Dehn twists (see \cite{MT}) and suggest a 3-dimensional description of the Goldman-Turaev Lie bialgebra (see \cite{Mas}). See also \cite{Tur2} for Fox pairings from knots in the 3-sphere. 

This paper develops the study of Fox pairings from the viewpoint of the group cohomology of Poincar\'{e} duality pairs. We fix subgroups $S_1, \dots, S_m \subset G$ and denote the union $\cup_{k=1}^m S_k $ by $\sh$. Roughly speaking, if the pair ($G,\sh$) satisfies a Poincar\'{e}-Lefschetz duality over $\K$, it is called a $\K$-Poincar\'{e} duality pair (see \S \ref{OOO} for the definition). In Section \ref{SS1}, we review group cohomology and show (Proposition \ref{Keypro}) that the set of Fox pairings is almost in a 1-1 correspondence with the set of the double cocycles, i.e., $Z^1( G,\sh; Z^1(G ;M)) $. Accordingly, in subsequent sections, we study the cohomology $H^1( G,\sh; H^1(G ;M)) $.

In Section \ref{OOO}, we first show (Proposition \ref{Keypro699}) that every Fox pairing of an $n$-dimensional Poincar\'{e} duality pair is null-cohomologous if $n>2$. Thus, we concentrate on 2-dimensional Poincar\'{e} duality pairs and compute the cohomology. The following is a typical example.
\begin{thm}[
{see Theorem \ref{Thm6}}]\label{main2} If the pair ($G,\sh$) is a 2-dimensional Poincar\'{e} duality pair over $\K$ and $\sh \neq \emptyset $, then the cohomology $H^1( G,\sh; H^1(G ;\K[G])) $ is isomorphic to $\K$.
\end{thm}
\noindent
We call a generator of the cohomology {\it the fundamental Fox pairing of $(G,\sh)$} (Definition \ref{def6th}). Furthermore, we examine the uniqueness of the fundamental Fox pairing from a certain condition; see Theorem \ref{dere35} and Corollary \ref{dere3335}. As examples, if $G$ is a surface group, then the fundamental Fox pairing is equal to the Fox pairing, as defined by Turaev \cite{Tur} (see Proposition \ref{de4r3e35}). 

We also explore Fox pairings of other groups. We first observe the fundamental Fox pairings of some orbifolds of dimension two (see Proposition \ref{d43e35}). Next, in Section \ref{S997} examines the conditions of constructing Fox pairings from 3-manifold groups and shows that some homotopy groups are obstacles to the construction; see Proposition \ref{ksoiw}. Furthermore, Section \ref{kkkk} examines the commutator subgroups of knot groups and presents a $\Q$-duality with a comparison to the work of Turaev \cite{Tur,Tur2}. In addition, as a generalization of Fox pairings, Section \ref{Se996} introduces higher Fox pairings of the $\K$-Poincar\'{e} duality pair of higher cohomological dimension, where we will show (Theorem \ref{Thm6ht}) that the set of higher Fox pairings is closely related to the relative homology $H_2(G,\sh ;\K)$. For example, if $G$ is a knot group in the 3-sphere, we point out a higher Fox pairing in the form of a Gysin map that is obtained from a Seifert surface; see Example \ref{ex339}. Furthermore, we also discuss the existence of higher Fox pairings for aspherical closed $n$-manifolds and the case $\sh = \emptyset$; see Proposition \ref{ex77733} and Example \ref{ex7733}.

To conclude, the above computations of the cohomology $H^n( G,\sh; H^n(G ;\K[G])) $ predict non-trivial (higher) Fox pairings, although it remains a problem for the future to describe these non-trivial Fox pairings concretely and to give their applications.

\

\noindent {\bf Conventional notation.} We always write $G$ for a group and $\K$ for a commutative ring. By $\K[G] $, we mean the group ring of $G$. Furthermore, we fix a finite family of subgroups $S_1, \dots, S_m \subset G$ such that $S_i \cap S_j=\{1\} $ for any $i \neq j$ (possibly, every $S_i $ is empty). We denote the union $\cup_{k=1}^m S_m $ by $\mathcal{S} $.

\section{Fox pairings from the viewpoint of group cohomology}
\label{SS1}
The purpose of this section is to suggest an approach to Fox pairings in terms of group (co)-homology.

We will begin by reviewing the (relative) group (co)-homology. Denote the group ring of a group $G$ over $\K$ by $\K[G]$. For $n \geq 1$, let $ F_n(G)$ be the free $\K[G]$-module with basis $\{ [g_1 | \cdots| g_n] ;g_i \in G\} $. Given a left $\K[G]$-module $A$, we define $C_n(G;A )$ to be $ A \otimes_{\K[G]} F_n(G) $ and the differential operator $\partial_n^l$ by
\[\!\!\!\!\!\! ( g_1^{-1} a) \otimes [ g_2| \cdots |g_{n}] + \!\!\!\!\! \sum_{i: \ 1 \leq i \leq n-1}\!\!\!\!\!\! (-1)^i a \otimes [ g_1| \cdots |g_{i-1}| g_{i} g_{i+1}| g_{i+2}|\cdots | g_n] +(-1)^{n} a \otimes [ g_1| \cdots | g_{n-1}] .\]
Then, we define the group homology $H_n(G; A)$ from this complex. 
The relative homology $H_n(G, \sh ;A)$ is defined to be the homology of the quotient complex $ C_n(G;A )/ \sum_{k=1}^m C_n(S_k;A )$. Dually, given a right $\K[G]$-module $B$, we define $C_n^r(G;B)$ to be $ F_n(G) \otimes_{\K [G]} B $ and another differential operator $\partial_n^r $ by
\[\!\!\!\!\!\! [ g_2| \cdots |g_{n}]\otimes b + \!\!\!\!\! \sum_{i: \ 1 \leq i \leq n-1}\!\!\!\!\!\! (-1)^i [ g_1| \cdots |g_{i} g_{i+1}| \cdots | g_n] \otimes b+(-1)^{n} [ g_1| \cdots | g_{n-1}] \otimes b g^{-1} .\]
We can define the homology groups $ H_n(G; B) $ and $ H_n(G,\sh; B)$ in the same fashion.

Dually, let $ C^n_l(G;A)$ be the module $\Hom (F_n(G) ,A)$ consisting of left $\K[G]$-module homomorphisms $f$. Furthermore, we define $ C^n_l(G,\sh ,A)$ to be the submodule consisting of such $f$'s satisfying $f(l_{i,1},\dots, l_{i,n} )=0$ for any $i \leq m$ and any $l_{i,1},\dots, l_{i,n} \in S_i$. For $f \in C^n(G;A)$, define $\delta^n_l(f)\in C^{n +1}(G;A)$ to be $f\circ \partial_{n+1}^l $. We denote by $ Z^n_l(G,\sh ;A)$ the submodule of $C^n_l(G,\sh ;A) $ consisting of $n$-cocycles. Dually, starting from a right $\K[G]$-module $B$, we can define the relative complex $ C^n_r(G,\sh ; B)$ in a parallel way to $C^n_l(G,\sh ; B) $.
\begin{exa}\label{ex1}
Consider the case $n=1$. A {\it left }(resp. a {\it right) derivation} $\partial: \K[G] \ra A$ is a $\K$-homomorphism satisfying $\partial (ab)= \partial (a) \mathrm{aug}(b)+ a \partial (b)$ (resp. $\partial (ab)= \partial (a) b+ \mathrm{aug}(a) \partial (b))$. Then, by the definition of $\delta^1_l$, $Z^1_l(G,A)$ is identified with the set of left derivations.
\end{exa}
\noindent
In addition, we can define the cap-product $\cap$; see, e.g., \cite{BE} for the definition. 

Next, we give some examples of Fox pairings:
\begin{exa}\label{deriex}
Let $M$ be a $\K[G]$-bimodule. If there are a $\K[G]$-bimodule homomorphism $ \nu: M \otimes_{\K[G] } M \ra M$ and a left derivation $D_l$ and a right one $D_r$ over $\K[G]$, the map $\K[G] \times \K[G] \ra M $ which takes $ a \otimes b$ to $ \nu( D_l(a) , D_r(b))$ is a Fox pairing. We denote the pairing by $\eta_{D_l\otimes D_r}$. For example, when $G$ is a free group $ \pi$ and $M=\K[ \pi]$, any Fox pairing is a sum of such Fox pairings arising from derivations (see \cite[Section 2.5]{MT}).

For $ c \in M$, the map that sends $ (a,b)$ to $ (\aug (a)-a) c (\aug(b)-b)$ is a Fox pairing, where $a,b \in \K[G]$. Such a Fox pairing is said to be {\it inner}. Furthermore, given a Fox pairing $\eta$ with $M= \K[G]$, the mapping $(g,h) \mapsto g \overline{\eta (h^{-1}, g^{-1})} h$ is also a Fox pairing, where the overline means the involution of $\K[G]$ defined by $\bar{g}=g^{-1}.$ We write $\eta^t$ for the mapping and call it {\it the transpose of $\eta$}. 
\end{exa}

Now we will give a characterization of every Fox pairing in terms of group cocycles (Proposition \ref{Keypro} below). Let $M$ be a $\K[G]$-bimodule and $\sh ,\sh' $ be collections of subgroups. Let us denote the set of Fox pairings by $\mathcal{FP}(G,M) $, which is canonically turned into a $\K $-module. Define a submodule of $\mathcal{FP}(G;M) $ by setting
$$ \{ \eta : \K[G] \otimes \K[G] \ra M : \ \mathrm{Fox \ }\mathrm{pairing} \ | \ \ \eta ( l\otimes g) =\eta ( g\otimes l' ) =0 \ \mathrm{for \ any \ }l\in \sh , l ' \in \sh', g\in G \ \}, $$
and denote it by $\mathcal{FP}(G,\sh,\sh';M) $. Take two copies, $G_1, G_2$, of $G$. Regarding $M$ as a right $\K[G_1]$-module, we can define the complex $ C_r^*(G_1;M)$. Since the left action of $G $ on $M$ gives rise to an action of $G$ on $C_r^*(G_1;M)$; thus, we can consider another complex $ C^*_l(G_2 ; Z_r^i(G_1 ; M))$. For a 1-cocycle $f \in Z^1_l(G_2 ; Z^1_r(G_1 ; M))$, we define a map $ \eta_{f}: G_2 \times G_1 \ra M $ by $ \eta_{f}(g,h):= (f(g))(h)$, which bilinearly extends to $ \K[G] \times \K[G] \ra M $ as a Fox pairing.

\begin{prop}\label{Keypro}
Let $M$ be a $\K[G]$-bimodule. The correspondence $f \mapsto \eta_f$ gives rise to a $\K$-module isomorphism,
$$ Z^1_l (G_2,\sh ; Z^1_r(G_1,\sh'; M)) \cong \mathcal{FP}(G,\sh,\sh ';M) . $$
Furthermore, if $\eta_f$ is inner, then $f $ lies in $ B^1_l (G_2,\sh ; B^1_r(G_1,\sh '; M))$.
\end{prop}
\begin{proof}
First, consider the case $\sh= \sh' = \emptyset $. By linearity, we have a bijection,
\begin{equation}\label{t333}\mathcal{FP}(G,M) \longleftrightarrow \{ \eta: G\times G \ra M \ | \ \eta \mathrm{ \ satisfies \ } \eqref{t1} \mathrm{ \ and \ } \eqref{t2} \}. \end{equation}
On the other hand, if we regard $C^n(G ; M)$ as the set $ \{ G^n \ra M\} $, we have canonical bijections,
$$ C_l^1(G_2 ; C_r^1(G_1 ; M)) \longleftrightarrow \mathrm{Map}( G_2 , \mathrm{Map}(G_1, M) ) \longleftrightarrow \mathrm{Map}( G_2 \times G_1, M ) . $$
We can readily see that the restriction on $ Z^1_l(G_2 ; Z^1_r(G_1 ; M))$ is onto the right-hand side of \eqref{t333} and that the composite of the bijections is equal to the required correspondence.

For the cases $ \sh \neq \emptyset$ and $\sh' \neq \emptyset$, we can easily check that the restriction on $\mathcal{FP}(G,\sh,\sh ' ;M) $ of the bijection above is the required isomorphism. 

The final claim can be readily shown by referring to the definitions of inner Fox pairings and the isomorphisms.
\end{proof}

\section{Fox pairings of duality groups of higher dimension }
\label{OOO}
As a consequence of Proposition \ref{Keypro}, it is reasonable to discuss the cohomology $ H^1(G_2,\sh ; H^1(G_1; M))$. In particular, we should focus on the case where $H^1(G_1; M) $ does not vanish. For applications, we will hereafter mainly consider the case $M=\K[G].$

Now let us discuss such (non)-vanishing cases with $M=\K[G]$. As the Shapiro lemma indicates (see, e.g., (6.4) in \cite[Section III. 6]{Bro}), if $G$ is finite, $H^1( G ;\K[G]) $ vanishes. So, we should consider groups of infinite order. As an example, consider duality groups, where a group $L$ is {\it a duality group} of dimension $n$, i.e., there are a $\K[G]$-module $D$ a homology $n$-class $ \mu \in H_n(L;D)$ such that the cap-product with $\mu$ gives an isomorphism $H^{* }(L;M ) \cong H_{n-*}(L;D \otimes_{ \K} M )$ for any coefficient $M$. Moreover, $G$ is a {\it virtual duality group of dimension }$n$, if there is a subgroup $L$ of $G$ of finite index that is a duality group of dimension $n$. Sections VIII. 8--10 in \cite{Bro} give examples.
\begin{prop}\label{Keypro6} 
If $G$ is a virtually duality group of dimension $n \neq 1$, then $ H^1( G ;\K[G]) $ vanishes as well.

\end{prop}
\begin{proof} 
Notice that $ H^{* }(G;\K[G] ) \cong H^{* }(L;\K[L] )$, by (6.4) in \cite[Section III. 6]{Bro}, and $H_*( L;\K[L]) \cong H_*( \mathrm{pt};\K)$, by the Shapiro lemma; thus, the duality implies $ H^1( G ;\K[G]) \cong H_{n-1}(L;(D \otimes \K)[L] ) \cong H_{n-1 }( \mathrm{pt};D \otimes \K ) =0$.
\end{proof}

As a result of Proposition \ref{Keypro6}, we should focus on duality groups of dimension one or Poincar\'{e} duality groups of low dimension in order to find non-trivial Fox pairings from duality groups.

At this point, we should review $\K$-Poincar\'{e} duality pairs in the sense of \cite[\S 6]{BE}. The pair $(G,\sh)$ is {\it a $\K$-Poincar\'{e} duality pair of dimension $n$} ($\K$-$\mathrm{PD}_n$ pair, for short) if there is a homology class $e \in H_n(G,\sh ; \K)$ with trivial coefficients such that the cap products
$$ e \cap \bullet: H^i(G,\sh;M) \lra H_{n-i}(G;M) , \ \ \ \ \ H^i(G;M) \lra H_{n-i}(G,\sh;M) $$
are isomorphisms for any $\K[G]$-module $M$. For example, for any orientable aspherical compact manifold $X$, if the inclusion $ \partial X \hookrightarrow X$ induces an injection $ \pi_1(\partial X) \hookrightarrow \pi_1(X)$, then the pair $(\pi_1(X) , \pi_1(\partial X)) $ is a $\K$-$\mathrm{PD}_n$ group for any ring $\K$; see \cite[Theorem 6.3]{BE}. Similarly to Proposition \ref{Keypro6}, we compute the first cohomology with $M=\K[G]$ as follows:
\begin{prop}\label{Keypro699}
Let the pair $(G,\sh)$ be a $\K$-$\mathrm{PD}_n$ pair.
\begin{enumerate}
\item If $n \geq 2$, then the relative $H^1( G,\sh ;\K[G]) $ vanishes.
\item If $ n \geq 3$, then the non-relative $H^1( G ;\K[G]) $ vanishes. If $n=2$ and $\sh = \emptyset$, then $H^1( G ;\K[G]) $ also vanishes.
\item If $n=2$ and $\sh \neq \emptyset $, then there is an isomorphism,
\begin{equation}\label{lll554}H^1( G ; \K[G]) \cong \Ker ( \bigoplus_{ k : 1 \leq k \leq m}\mathrm{aug}: \bigoplus_{ k : 1 \leq k \leq m} \K[G/ S_k] \ra \K ). \end{equation}
Here, if $H^1( G ; \K[G]) $ is the cohomology of $Z_r( G ; \K[G])$, which can be regarded as a left $\K[G]$-module, the isomorphism \eqref{lll554} holds in the sense of left $\K[G]$-modules.
\end{enumerate}
\end{prop}
\begin{proof} 
By duality and the Shapiro lemma, $H^1( G,\sh ;\K[G]) \cong H_{n-1} (G;\K[G])\cong H_{n-1} (\mathrm{pt};\K)= 0$ for $n \geq 2$. Next, $H^1( G ;\K[G]) \cong H_{n-1} (G,\sh;\K[G])$ is isomorphic to $H_{n-2} (\sh;\K[G])$ by the homology long exact sequence and $ H_{*} (G,\K[G])=0$ for $* \geq 1$ and $n \geq 3$. By the Shapiro lemma again, the homology $H_{n-2} (\sh;\K[G])$ is zero; hence, we have proven the second claim. A similar discussion holds for the case of $n=2$ and $\sh = \emptyset$.

Finally, to prove \eqref{lll554}, consider the homology long exact sequence
$$0= H_{1}( G ;\K[G]) \ra H_{1}( G, \sh ;\K[G]) \ra H_{0}( \sh ;\K[G])\stackrel{\mathrm{inc}_*}{\lra} H_0 ( G ;\K[G])\ra 0. $$
By the Shapiro lemma again, the third and fourth terms are computed as
$$ H_{0}( \sh ;\K[G]) \cong \bigoplus_{ k : 1 \leq k \leq m} \K[G/ S_k] , \ \ \ H_0 ( G ;\K[G])\cong H_0( \mathrm{pt};\K) \cong \K, $$
and the map $\mathrm{inc}_*$ coincides with the augmentation map. Therefore, by duality again, the cohomology $H^1( G ; \K[G]) \cong H_1 (G,\sh ;\K[G]) $ is isomorphic to the kernel $\Ker (\aug)$, as required. Furthermore, if $M$ is also regarded as a left $\K[G]$-module, it is not hard to check that the above isomorphisms are left $\K[G]$-module homomorphisms. 
\end{proof}
Consequently, we shall focus on the case $n=2$ and the non-relative cohomology $H^1( G ; \K[G]) $.

\section{Fox pairings of Poincar\'{e} duality pairs of dimension two}
\label{OOO3}
In this section, we will focus on Fox pairings of Poincar\'{e} duality pairs of dimension two. The cohomology consisting of Fox pairings is computed as follows (see \S \ref{SS2346} for the proof):
\begin{thm}\label{Thm6}
Let the pair $(G,\sh)$ be a $\K$-$\mathrm{PD}_2$ pair with $\sh \neq \emptyset$. Then, there is an isomorphism,
\begin{equation}\label{l6654} H^1(G_2,\sh ; H^1(G_1; \K[G])) \cong \K.\end{equation}

Meanwhile, concerning the non-relative cohomology, there is an exact sequence,
\begin{equation}\label{lll5454} 0 \ra \K \lra H^1(G_2 ; H^1(G_1; \K[G])) \lra \frac{ \Ker (\aug :\oplus_{k=1}^m \K [G/S_k ] \ra \K) }{\{ a - \zeta \cdot a \}_{\zeta \in \mathcal{S}}} \lra 0 .\end{equation}
\end{thm}

Theorem \ref{Thm6} implies that, while there are infinitely non-nullcohomologous Fox pairings, there is a unique Fox pairing as a basis of $ H^1(G_2,\sh ; H^1(G_1; \K[G])) \cong \K $ in \eqref{l6654}. Thus, we define
\begin{defn}\label{def6th} The {\it fundamental Fox pairing} is a Fox pairing $G^2 \ra \K[G]$, which represents a basis of $ H^1(G_2,\sh ; H^1(G_1; \K[G])) \cong \K $. 
\end{defn}
Furthermore, we will discuss uniqueness of Fox pairings under some conditions.
\begin{thm}
[{See Section \ref{SS2346} for the proof}]\label{dere35} Let the pair $(G,\sh)$ be a $\K$-$\mathrm{PD}_2$ pair. Then, a Fox pairing $\eta: G^2 \ra \K[G]$ and $a_{s} \in \K[G]$ uniquely associated with $s \in \sh $ exist such that
\begin{equation}\label{t5} \eta( s , g)= a_{s}( 1-g) \ \ \mathrm{for \ any \ } g \in G, \end{equation}
and $\eta$ represents a basis of $H^1(G_2,\sh ; H^1(G_1; \K[G])) \cong \K $ in \eqref{l6654}

Furthermore, if $\sh \cong \Z$ and $\K$ is a field, we can choose a Fox pairing $\eta$ satisfying
\begin{equation}\label{t57} \eta( \zeta , g)= 1-g \ \ \mathrm{for \ any \ } g \in G, \end{equation}
where $\zeta$ is a generator of $ S_1 \cong \Z$.
\end{thm}
\begin{cor}\label{dere3335}
Let $(G,S)$ be a $\K$-$\mathrm{PD}_2$ pair satisfying $S \cong \Z$, and $\K$ be a field. Let $\eta$ be the Fox pairing satisfying \eqref{t57}. Then, any group isomorphism $f : G \ra G$ satisfying $f(S)=S$ preserves $\eta$. More precisely, $\eta ( f(g), f(h))= f( \eta(g,h))$ for any $g,h \in G$.
\end{cor}
\begin{proof} 
We can easily check that the map $ \eta' : G^2 \ra \K[G]$ which sends $(g,h)$ to $f^{-1}(\eta ( f(g), f(h)))$ is a Fox pairing and satisfies \eqref{t57}. Thus, $\eta =\eta'$ by uniqueness. Namely, $f $ preserves $\eta$.
\end{proof}
In addition, we will discuss representatives of the cohomology classes.
\begin{thm}\label{Thm688}
Let $(G,S)$ be a $\K$-$\mathrm{PD}_2$ pair with $\sh \neq \emptyset$. Then, every cohomology class of the cohomology groups in \eqref{l6654} and \eqref{lll5454} is represented by a Fox pairing. Moreover, if $\K$ is a PID, then every cohomology class of these cohomologies is represented by a sum of the Fox pairings like in Example \ref{deriex}.
\end{thm}
The proof is in Appendix \ref{SS2}. Here, we should note that the fundamental Fox pairing is not always represented by a cocycle in the relative $Z^1(G_2,\sh ; Z^1(G_1; M)) $, but is always represented by a cocycle in the non-relative $Z^1(G_2 ; Z^1(G_1; M)) $; see Proposition \ref{de4r3e35} for examples.

\section{Proofs of Theorems \ref{Thm6} and \ref{dere35}}
\label{SS2346}
A hasty reader may skip this section.

\begin{proof}[Proof of Theorem \ref{Thm6}]Let the coefficient $M$ be the kernel in 
\eqref{lll554}, i.e., $M= \Ker(\aug: \oplus_{k=1}^m \K[G/S_k ] \ra \K]) $. Consider the long exact sequence,
\[ H_2( G; \K ) \lra H_1( G; M) \lra \oplus_{k=1}^m H_1( G;\K [G/S_k ]) \xrightarrow{ \ \aug_* \ } H_1( G;\K ) \ra \]
\[ \ \ \ \ \ \ \ \ \ \ \ \ \ \ \ \ \ra H_0( G; M) \lra \oplus_{k=1}^m H_0( G;\K [G/S_k ])\xrightarrow{ \ \aug_* \ } H_0( G;\K ) \lra 0 \ \ \ \ \mathrm{(exact)}.\]
Notice that $H_i( G;\K [G/S_k ]) \cong H_i( S_k ;\K )$, by the Shapiro lemma. 
Thus, the long sequence reduces to the exact sequence,
\begin{equation}\label{lll45i} 0 \ra H_1( G; M) \ra \oplus_{k=1}^m H_1( S_k;\K ) \xrightarrow{ \ \mathrm{aug}_* \ } H_1( G;\K ) \ra H_0( G; M) \ra \K^m \lra\K \ra 0 ,\end{equation}
and the map $\mathrm{aug}_* $ is equal to the induced map $H_1( \sh ;\K ) \ra H_1( G;\K ) $ with trivial coefficients from the inclusion $S_k \hookrightarrow G$. Also notice that, by the homology long exact sequence,
$$ H_2(G;\K)=0 \lra H_2(G, \sh;\K) \lra \oplus_{k=1}^m H_1( S_k;\K ) \lra H_1(G; \K),$$
the kernel of the map $\mathrm{aug}_* $ is isomorphic to $H_2(G,\sh;\K) \cong H^0(G;\K) \cong \K$. In summary, since $H^1(G,\sh;M ) \cong H_1(G;M)$, we can readily obtain \eqref{l6654} from \eqref{lll45i}.

To prove the sequence \eqref{lll5454} in the latter claim, let us write $B$ for $ H^1(G_1;\K[G] )$ for short. Consider the homology long exact sequence,
\begin{equation}\label{ll3l4i} \oplus_{k=1}^m H_1( S_k; B) \lra H_1(G_2; B) \lra H_1(G_2, \sh ;B)\lra \oplus_{k=1}^m H_0( S_k;B ) \lra H_0(G_2;B) . \end{equation}
We will observe each term. Since each $S_i$ is a $\K$-$\mathrm{PD}_1$ group (see \cite[Theorem 4.2]{BE}), the first term $H_1( S_k;B) $ is isomorphic to the invariant part $H^0( S_k;B) =B^{S_k}$, which is zero. The last term can be shown to be zero by checking the (co)-invariant part of $B$. We can invoke the claim above to show that the second term is $\K$, and the fourth term is the coinvariant of $ B= \Ker(\aug: \oplus_{k=1}^m\K[G/ S_k ] \ra \K)$. Since the third term in \eqref{ll3l4i} is $H^1( G_2; H^1(G_1;\K[G])) $ by duality, the exact sequence \eqref{ll3l4i} turns out to be the required one \eqref{lll5454}.
\end{proof}

Next, we will prove Theorem \ref{dere35}. For this, we will need a lemma.
\begin{lem}\label{Keylem4434}
There is no non-trivial derivation $D: G\ra \K[G ]$ such that $ D(\zeta)=0$ for any $\zeta \in \sh $.
\end{lem}
\begin{proof}
This follows straightforwardly from $H^1(G, \sh ;\K[G ] ) =0 $ by Proposition \ref{Keypro699} and $ B^1(G, \sh ;\K[G ] ) =0 $ by definition.
\end{proof}

\begin{proof}[Proof of Theorem \ref{dere35}]
Let $\eta$ be the fundamental Fox pairing, which lies in $Z^1(G, Z^1(G;\K[G]))$ by Theorem \ref{Thm688}. From the definition of $ H^1( G, \sh ;H^1( G;\K [G])) $, for any $s \in \mathcal{S}$, $\eta(s ,\bullet )$ is null-cohomologous in $Z^1_r(G_1;\K[G])$; thus, there is $ a_s \in \K [G] $ such that $\eta( s ,h)= a_s(1-h^{-1})$ as required.

Next, we will show uniqueness. Suppose there is another such Fox pairing $ \eta '$. Then, $ \eta-\eta '( \zeta, a)=0$ for any $a \in G$ and $\zeta \in \mathcal{S}$. Thus, the transpose of $ \eta-\eta' $ lies in $ Z^1(G ; Z^1(G,\sh ;\K [G] ) ) =0$ by Lemma \ref{Keylem4434}. That is, $\eta = \eta'$, as desired.

We will prove the final statement. Suppose $m=1$, and $S_1 \cong \Z$. By Lemma \ref{kl992} below, there is a left derivation $D:G \ra \K [G]$ satisfying $D(\zeta)= a_{\zeta}- \aug(a_{\zeta}) + (1-\zeta)b_{\zeta} $ for some $b_{\zeta}\in \K[G]$. Let us define $k_{\eta} \in \K$ to be $ \aug(a_{\zeta}) $ and a Fox pairing $\eta'' $ by setting
$$\eta''(g,h)= \eta( g ,h)-D(g) (1-h) - (1-g)b_{\zeta} (1-h).$$
Then, if we can show that $k_{\eta} \neq 0$, then $k_{\eta}^{-1} \eta''$ implies the existence of the Fox pairing, since $\K$ is a field by assumption.

Thus, let us show that $k_{\eta} \neq 0$. Suppose $k_{\eta} =0$. Then, $\eta \in Z^1(G,\sh ; Z^1(G; \K [G] ) )$ by Proposition \ref{Keypro}. Then, the transposed $^t \eta$ lies in $ Z^1(G ; Z^1(G,\sh;\K [G] ) )$, which is zero by Lemma \ref{Keylem4434}. Thus, $\eta $ is zero and does not represent a generator of $ H^1( G,\sh ;H^1( G;\K [G])) $, which is a contradiction.
\end{proof}

\begin{lem}\label{kl992}
Suppose $m=1$ and $S_1 \cong \Z = \langle \zeta \rangle.$ Then, for any non-zero $ b \in \Ker(\aug) \subset \K[G]$, a left derivation $D : G \ra \K[G]$ and $c \in \K[G]$ exist such that $D(\zeta)=b- (1- \zeta)c$.
\end{lem}
\begin{proof} 
Consider the (co)-homology long exact sequences:
$${\normalsize
\xymatrix{ H^1( G,\sh;\K[G]) \ar[d]^{\cong} \ar[r]&H^1( G;\K[G]) \ar[d]^{\cong} \ar[r] & H^1( \sh;\K[G]) \ar[d]^{\cong} & \\
H_1(G;\K[G])=0 \ar[r] &H_1( G,\sh;\K[G]) \ar[r] & H_0( \sh;\K[G]) \ar[r]^{\!\!\!\!\!\!\!\!\!\! (\mathrm{inc})_*} & H_0( G;\K[G]) \cong \K \lra 0.
}}
$$
Here, the vertical maps are the cap products in duality. By reconsidering the proof of the isomorphism $ H_0 ( \sh;\K[G]) \cong \K[G]/(1-\zeta) $ in the Shapiro lemma, the map $(\mathrm{inc})_*$ coincides with the augmentation map. Hence, diagram chasing leads to the conclusion that, for any derivation $ D': L \ra \K[G]$ such that $D'(\zeta) =b$, there is an extension $D$ of $D'$ such that $D(\zeta)=b$ up to coboundary. Namely $D(\zeta)=b- (1- \zeta)c$ for some $c\in \K[G] $.
\end{proof}

\section{Fox pairings of surface groups and 2-dimensional orbifolds}
\label{SS234446}
We will examine surface groups as a typical example of Poincar\'{e} pairs of dimension two and discuss a theorem in \cite{Tur}. Take a connected compact surface $\Sigma$ with boundaries. Let $G$ be $\pi_1(\Sigma)$ and $ \sh$ be $\pi_1( \partial \Sigma)$. If $ \Sigma$ is orientable, we let $\K$ be $\Z$; otherwise, we let $\K$ be $\Z/2$. Then, by Poincar\'{e}-Lefschetz duality, the pair $(G,\sh) $ is a $\K$-$\mathrm{PD}_2$ pair.

Following \cite[\S\S 1.4--1.6]{Tur} (see also Section 3.1 in \cite{Mas}), we can define a Fox pairing $\eta_{\Sigma}: G^2 \ra \K[G]$, as follows. Fix a base point $* \in \partial \Sigma $, and $\bullet , \dagger \in \partial \Sigma $ be additional points such that $ \bullet <* < \dagger$ in the same component of $\partial \Sigma $. Let $I$ be the interval $[0,1]$, $\nu_{\bullet *}$ the path from $\bullet$ to $*$, and $\nu_{ * \dagger}$ the path from $* $ to $\dagger$. Let $ \overline{ \nu}_{* \bullet }$ and $\overline{\nu}_{ * \dagger }$ be the respective paths with opposite orientations. For $a,b\in \pi_1( \Sigma, *)$, we choose a loop $\alpha$ based at $\bullet$ such that $\overline{ \nu}_{* \bullet } \alpha \nu_{ \bullet *} $ represents $a$ and a loop $\beta$ based at $\dagger$ such that $\overline{ \nu}_{* \dagger } \beta \nu_{ \dagger *} $ represents $b$. Here, we may assume that $\alpha$ and $\beta$ are generic immersions and that they cross transversally. Then,
$$ \eta_{\Sigma}^{\rm pre}(a,b):= \sum_{p \in \alpha \cap \beta} \varepsilon_{p}(\alpha, \beta) \overline{ \nu}_{* \bullet } \alpha_{\bullet p} \beta_{p \dagger} \overline{\nu}_{\dagger *} \in \K[\pi_1(\Sigma)]. $$
Here, the sign $\varepsilon_{p}(\alpha, \beta) \in \{ \pm 1\}$ is defined to be $+1$ if and only if a unit tangent vector of $\alpha$ at $p$ followed by a unit tangent vector of $\beta$ at $p$ gives a positively-oriented frame of $\Sigma.$ If $\Sigma$ is not orientable, $\varepsilon_{p}(\alpha, \beta) $ is always one. The bilinear extension of $ \eta_{\Sigma}^{\rm pre}$ is written as $ \eta_{\Sigma}: \K[G]^{2} \ra \K[G]$ and is known to be a Fox pairing. Regarding the generator $ \nu \in \sh \cong \sqcup^q \Z$, it is known that $\eta_{\Sigma}(\nu, G) \subset (\nu -1) \Z[G]$ (see Page 232 in \cite{Tur}). Thus, $\eta_{\Sigma} $ can be regarded as a 1-cocycle in $Z^1( G_2, \sh ; H^1( G_1 ;\K[G])) \cong \K$ as in Proposition \ref{Keypro}.
\begin{prop}\label{de4r3e35}
The Fox pairing $\eta_{\Sigma}$ represents the fundamental Fox pairing of $\pi_1(\Sigma)$. 
\end{prop}
\begin{proof}
The augmentation $\aug: \K[G] \ra \K$ as the transformation coefficient induces $\K \cong H^1( G_2, \sh ; H^1( G_1 ;\K[G])) \ra H^1( G_2, \sh ; H^1( G_1 ;\K) )\cong H^1( G_2,\sh;\K) \otimes H^1( G_1 ;\K) )$. Notice that $\aug(\eta^{\rm pre}_{\Sigma}(a,b))\in \K $ is equal to the intersection number of $a$ and $b$ by definition. Therefore, $\aug_* (\eta_{\Sigma} ) $ is equal to the intersection form $I_{\Sigma} $ on $H^1(\Sigma ; \K) $. Since all of the entries of $I_{\Sigma} $ lie in $\{ -1,0,1\} $, $\eta_{\Sigma} $ must be a generator of $ H^1( G_2, \sh ; H^1( G_1 ;\K[G])) \cong \K$, as required.
\end{proof}

Furthermore, let us show that uniqueness immediately follows from Theorem \ref{dere35} and make a comparison with \cite{Tur}. 
\begin{cor}\label{der3e35}
Let the pair $(G,\sh)$ be $(\pi_1(\Sigma), \pi_1(\partial \Sigma) )$, as above. Let $ \zeta_1, \dots, \zeta_m \in \pi_1(\partial \Sigma) \cong \sqcup^m \Z$ be generators. If a Fox pairing $\eta: G^2 \ra \K[G]$ satisfies
\begin{equation}\label{t5888} \eta( \zeta_k , g)= \eta_{\Sigma} (\zeta_k, g) \ \ \mathrm{for \ any \ } g \in G, \ k\leq m, \end{equation}
then the pairing $ \eta$ is equal to $\eta_{\Sigma} $. Such a pairing is a representative of a generator of \eqref{l6654}.
\end{cor}
\noindent
A stronger version of this corollary is given in \cite[Theorem I and Corollary]{Tur} for the case that $ \Sigma$ is orientable.

Next, we present a procedure for expressing the Fox pairings of a two-dimensional orbifold under some conditions. Let $ \Sigma_{\rm orb}$ be a compact 2-dimensional orbifold with circle boundaries. Suppose that a finite group $\Gamma $ and a surjective homomorphism $f: \pi_1(\Sigma_{\rm orb}) \ra \Gamma $ satisfying $f (\pi_1(\partial \Sigma_{\rm orb}))=\{ 1\}$ exist such that the associated covering space is an oriented surface $\Sigma$ and admits a group extension,
\begin{equation}\label{t555} 0 \ra \pi_1(\Sigma) \lra
\pi_1(\Sigma_{\rm orb}) \stackrel{f}{\lra} \Gamma \lra 0. \end{equation}
Then, by following the discussion in \cite[Theorem 7.6]{BE} or \cite[Proposition 5.3]{Fow}, we can easily establish the following proposition.
\begin{prop}\label{d43e35999}
Under the above assumption, $ ( \pi_1( \Sigma_{\rm orb}), \pi_1(\partial \Sigma_{\rm orb}))$ is a $\Q$-$ \mathrm{PD}_2$ pair.
\end{prop}
Moreover, as in \cite[Section III.6]{Bro}, the Shapiro lemma yields an isomorphism,
$$ \mathcal{I} : H^1(\pi_1( \Sigma_{\rm orb}) ; \Q [\pi_1( \Sigma_{\rm orb})]) \cong H^1(\pi_1( \Sigma ) ; \Q [\pi_1( \Sigma)]). $$
Thus, the restriction map $ \pi_1(\Sigma) \hookrightarrow \pi_1(\Sigma_{\rm orb}) $ yields a homomorphism, 
$$\mathcal{I} \otimes \mathrm{res}_* : H^1(\pi_1( \Sigma_{\rm orb}), \pi_1(\partial \Sigma_{\rm orb}) ; H^1(\pi_1( \Sigma_{\rm orb});\Q[\pi_1( \Sigma_{\rm orb})])) \lra H^1(\Sigma, \partial \Sigma ; H^1( \Sigma ;\Q[\pi_1( \Sigma)])) . $$
\cite[Proposition III.10.4]{Bro} uses transfer maps to show that $\mathcal{I} \otimes \mathrm{res}_*$ is injective; since the image and domain of $\mathcal{I} \otimes \mathrm{res}_* $ are computed as $\Q$ by Theorem \ref{Thm6}, $\mathcal{I} \otimes \mathrm{res}_*$ is an isomorphism. As is known \cite{Bro}, the inverse map is constructed by a transfer map $\mathrm{Tr}$ with respect to $\mathrm{res}_*$. Hence, if we explicitly describe the transfer map on a chain level as in \cite[Section III.9]{Bro}, we can express $ \mathrm{Tr}_*( \eta_{\Sigma})$ as well. To summarize:
\begin{prop}\label{d43e35}
Under the above assumption, the fundamental Fox pairing of the orbifold group $\pi_1( \Sigma_{\rm orb}) $ over $\Q$ is represented by $c_{\Sigma} \mathrm{Tr}_*( \eta_{\Sigma}) $ for some $c_{\Sigma} \in \Q^{\times}$.
\end{prop}
We conclude this section by mentioning (co)-bracket structures. According to the discussion in \cite[Remark 7.4]{MT}, such a Fox pairing defines a bracket $ [,]: \K[ \mathcal{G}/\mathrm{conj}] \otimes \K[ \mathcal{G}/\mathrm{conj}] \lra \K[ \mathcal{G}/\mathrm{conj}] $, where $ \mathcal{G}= \pi_1( \Sigma_{\rm orb})$, and $ \mathcal{G}/\mathrm{conj}$ is the set of conjugacy classes of $ \mathcal{G}$. It might be an interesting to see whether the bracket defines a Lie algebra structure on $\K[ \mathcal{G}/\mathrm{conj}] $. Incidentally, Appendix \ref{fff} discusses the existence of the cobracket on $\K[ \mathcal{G}/\mathrm{conj}] $.

\section{Fox pairings of 3-manifold groups}
\label{S997}
In this section, we focus on Fox pairings derived from 3-manifold groups. Let $X$ be a connected orientable compact 3-manifold, where the boundary $ \partial X$ is non-empty. Let $\iota$ be the inclusion $\iota:\partial X \hookrightarrow X$. By Proposition \ref{Keypro699}, if $X$ is aspherical and $\iota_*: \pi_1(\partial X ) \ra \pi_1(X ) $ is injective, then every Fox pairing is null-cohomologous. Thus, we shall consider non-aspherical 3-manifolds and show (Propositions \ref{ksoiw} and \ref{kso04w}) that non-trivial Fox pairings of $\pi_1(X)$ are derived from the homotopy groups $\pi_*(X)$ and the (non)-injectivity of $\iota_*$.

We will begin by discussing the first cohomology groups of $\pi_1(X)$. As a basic fact (see \cite[Chapter II]{Bro}), recall the canonical isomorphism between the singular cohomology and the group cohomology of degree one, that is, $H^1(\pi_1(X),\pi_1(\partial X );M) \cong H^1(X,\partial X;M)$ for any coefficient $M$.
\begin{prop}
[{cf. Proposition \ref{Keypro699}}]\label{ksoiw} Let $G$ be $\pi_1(X)$, and $\sh $ be $\pi_1(\partial X)$ as above. Assume that $ \mathrm{Im}(\iota_*)$ is not zero in $\pi_1(X)$. Then, the relative $H^1( X, \partial X; \K [G]) $ is isomorphic to $\pi_2(X) \otimes \K$; in particular, the following isomorphism holds:
\begin{equation}\label{t77333}H^1(\pi_1(X); H^1( X, \partial X; \K [G]) ) \cong H^1(\pi_1(X);\pi_2(X)\otimes_{\Z} \K) . \end{equation}
Furthermore, there is an exact sequence,
$$ 0 \lra \pi_2(X) \otimes_{\Z} \K \lra H^1 (X ; \K[G]) \lra H_1( \partial X ;\K[G]) \lra 0. $$
\end{prop}
\begin{proof}
Let $\widetilde{X}$ be the universal cover of $X$. Then, we immediately prove the former claim by
$$H^1( X, \partial X; \K [G]) \cong H_2( X; \K [G]) \cong H_2( \widetilde{X} ;\K) \cong \pi_2(X) \otimes \K. $$
Here, the three isomorphisms are obtained by invoking duality, Shapiro lemma, and Hurewicz theorem, respectively. Next, consider the homology long exact sequence,
\begin{equation}\label{kkl6} H_2(\partial X; \K [G] ) \ra H_2( X; \K [G] ) \ra H_2( X, \partial X; \K [G] ) \ra H_1(\partial X; \K [G] ) \ra H_1( X; \K [G] ) .\end{equation}
By duality, the first term is $H^0(\partial X; \K [G] )$, which is zero by assumption, and the last term vanishes because of $ H_1( \widetilde{X}; \K) =0.$ Hence, the required sequence is nothing but the dual of \eqref{kkl6}.
\end{proof}
Concerning Fox pairings, we shall focus on the cohomology in \eqref{t77333}. Let $ P_n X \ra P_{n-1}X \ra \cdots \ra P_0X$ be the Postnikov tower of $X$ (see, e.g., \cite[Section 4.3]{Mcc} for the definition). The cohomology is estimated as follows:
\begin{prop}\label{kso04w}
In the above notation, there is an exact sequence, 
\begin{equation}\label{kkl5} 0 \ra H^1(\pi_1(X);\pi_2(X) \otimes \K ) \ra H^4(\pi_1(X);\K)\ra H^4( P_2(X);\K),
\end{equation}
and there is a surjection from the invariant part $ (\pi_3(X) \otimes \K)^{\pi_1(X)}$ to the last term $ H^4( P_2(X);\K)$.
\end{prop}
\begin{proof} 
By the Leary-Serre cohomology spectral sequence of the Postnikov tower (see \cite[Lemma $8^{bis}.27$]{Mcc} for the details of the homological one), we have the exact sequences,
\begin{equation}\label{kkl600}H^3(P_2(X);\K)\ra H^1(\pi_1(X);\pi_2(X) \otimes \K ) \ra H^4(\pi_1(X);\K)\ra H^4( P_2(X);\K) ,\end{equation}
\begin{equation}\label{kkl8} 0 \ra H^{n}(P_{n-1}(X);\K) \ra H^n(X;\K) \ra (\pi_n(X) \otimes \K)^{\pi_1(X)},
\end{equation}
\begin{equation}\label{kkl9} (\pi_3(X) \otimes \K)^{\pi_1(X)} \ra H^4( P_2(X);\K) \ra H^4( P_3(X);\K) .\end{equation}
Then, the desired \eqref{kkl5} is readily due to \eqref{kkl600} and \eqref{kkl8}. Furthermore, the required surjection arises from \eqref{kkl9} and \eqref{kkl8} with $n=3$.
\end{proof}
To summarize, to compute the cohomology $ H^1(X, \partial X; H^1( X; \K [G]) )$, it is important to compute $H^4(\pi_1(X);\K)$ and the homotopy groups $\pi_*(X)$. Similarly, the other cohomology groups,
$$H^1(\pi_1(X), \pi_1(\partial X); H^1( X; \K [G]) ) , \ \ \ \mathrm{and} \ \ \ \ H^1(\pi_1(X); H^1( X; \K [G]) ) $$
can be estimated from spectral sequences. However, by the loop theorem and the sphere theorem, most 3-manifolds with infinite fundamental groups are aspherical and satisfy the injectivity of $\iota_*: \pi_1(\partial X ) \ra \pi_1(X ) $. Furthermore, we should mention that in cases where $\pi_2(X)$ is infinitely generated, it might be hard to find non-trivial Fox pairings using the above propositions. 

\section{Duality of commutator subgroups of knot groups }
\label{kkkk}
Some Fox pairings of the commutator subgroups of knot groups are discussed in \cite[Theorem E]{Tur2} and \cite[Appendix 3]{Tur}. Here, we make a conjecture (Problem \ref{Thm69933}) and state and prove Proposition \ref{3dere44hm5} that relates to this conjecture.

Let $ K$ be a knot in an integral homology 3-sphere $\Sigma$. Since $H_1( \Sigma \setminus K ;\Z) \cong \Z$, we have an infinite cyclic cover, $\tilde{E}_K ,$ of $\Sigma \setminus K$. If $\Sigma \setminus K $ is irreducible, the knot group pair $( \pi_1( \Sigma \setminus K), \Z \times \Z)$ is known to be a $\mathrm{PD}_3$-pair; thus, by Proposition \ref{Keypro699}, the Fox pairings of $\pi_1( \Sigma \setminus K)$ are trivial. In contrast, we will focus on the commutator subgroup $[ \pi_1( \Sigma \setminus K), \pi_1( \Sigma \setminus K) ]=\pi_1(\tilde{E}_K)$ and pose a question.

\begin{prob}\label{Thm69933}
Let $G$ be $\pi_1( \tilde{E}_K )$, and $\sh $ be the subgroup $ \pi_1(\partial \tilde{E}_K ) $. Then, is the pair ($G,\sh$) a $\Q$-$\mathrm{PD}_2$ pair? 
\end{prob}
The duality theorem only with trivial coefficients $\Q$ is known to the Milnor pairing \cite{Mil}. However, if the problem is positively solved, we can discuss a generalization of \cite[Theorem I ]{Tur} and \cite[Theorem E]{Tur2} as follows.
\begin{prop}\label{3dere44hm5}
Let $G$ be $\pi_1( \tilde{E}_K )$, and $\sh$ be the subgroup $ \pi_1( \partial \tilde{E}_K )\cong \langle \zeta \rangle $. Suppose that the pair ($G,\sh$) is a $\Q$-$\mathrm{PD}_2$ pair like Problem \ref{Thm69933}.

Then, there exists uniquely a Fox pairing $ \eta : \Q[G] \otimes \Q[G] \ra \Q[G]$ satisfying
\begin{equation}\label{t588888} \eta( \zeta , g)= 1-g \ \ \mathrm{for \ any \ } g \in G. \end{equation}
Furthermore, if $\pi_1( \tilde{E}_K ) $ is residually nilpotent, then
\begin{equation}\label{t457} \eta( g,h)+ \eta^t ( g,h )= (1-g)(1-h) \ \ \mathrm{for \ any \ } g ,h\in G . \end{equation}
\end{prop}
We give a comparison with existing results. If $M$ is replaced by a completed module of $\K[G] $, the same statement in the knot case is also shown in \cite[Theorem E]{Tur2}, and is not used by homology algebra. Here, the point is that our statement discusses before completing the module $\K[G] $.

Finally, we will give the proof of Proposition \ref{3dere44hm5}. For a group $\mathcal{G} $, let $ I \subset \K[\mathcal{G} ]$ be the augmentation ideal. Consider the inverse limit $\widehat{\K}[\mathcal{G} ]:= \lim_n \K[\mathcal{G} ]/ I^n.$ As is classically known \cite[Appendix A.]{Qui}, this $\widehat{\Q}[\mathcal{G} ]$ is isomorphic to the group ring $\widehat{\Q}[\widehat{\mathcal{G} } ] $, where $\widehat{\mathcal{G} }$ the Malcev completion of $\mathcal{G} $, that is, $\widehat{\mathcal{G} }= \lim \mathcal{G} /\mathcal{G} _n \otimes \Q $, where $\mathcal{G} \supset \mathcal{G}_2 \supset \mathcal{G}_3 \supset \cdots$ is the lower central series of $\mathcal{G}$, and $ \mathcal{G} /\mathcal{G} _n \otimes \Q $ is the rationalization of $ \mathcal{G} /\mathcal{G} _n $. Note that the canonical map $\mathcal{G} \ra\widehat{\mathcal{G} } $ is injective if and only if $\mathcal{G}$ is torsion free and residually nilpotent. If so, the induced map $\Q[\mathcal{G} ] \ra \Q[\widehat{\mathcal{G} }] $ is injective. In addition, Turaev shows \cite[Theorem E]{Tur2} that, if $G = \pi_1(\tilde{E}_K )$, there is a unique Fox pairing $ \hat{\eta}: \Q[G] \times \Q [G] \ra \widehat{\Q }[ \widehat{G}] $ satisfying
\begin{equation}\label{t587} \hat{\eta}( \zeta , g)= 1-g, \ \ \ \ \hat{\eta}( g,h)+ \hat{\eta}^t ( g,h )= (1-g)(1-h) \ \ \mathrm{for \ any \ } g ,h\in G . \end{equation}
\begin{proof}
[Proof of Proposition \ref{3dere44hm5}] The former statement is immediately obtained from Theorems \ref{Thm6} and \ref{dere35}. To prove the latter one, notice that the knot group $\pi_1( \Sigma \setminus K)$ is torsion free. Thus, so is $G$. Therefore, we have injections $\Q[G] \hookrightarrow \Q[\widehat{G } ]\hookrightarrow \widehat{\Q} [\widehat{G } ]$. Hence, \eqref{t587} implies \eqref{t457}.
\end{proof}

\section{Higher Fox pairings of Poincar\'{e} duality pairs }
\label{Se996}
Now let us introduce higher Fox pairings and study the cohomology of higher Fox pairings of Poincar\'{e} duality pairs. We fix integers $n,m\in \mathbb{N}$ and a $\K[G]$-bimodule $M$.

For a $\K$-bilinear map $\eta: \K[G]^{\otimes m} \times \K[G]^{\otimes n} \ra M $ and $ a \in \K[G]^{\otimes m} , b \in \K[G]^{\otimes n} $, we can define a right evaluation $ \eta_{b}: \K[G]^{\otimes m} \ra M $ that sends $ c $ to $\eta (c,b)$ and a left evaluation $_a \eta : \K[G]^{\otimes n} \ra M $ that sends $d $ to $\eta (a,d)$. The maps $_a \eta$ and $\eta_b $ can be regarded as maps $G^m \ra M$ and $G^n \ra M$, respectively.
\begin{defn}\label{def6th4}
Let $\mathcal{S}$ and $\mathcal{S}'$ be finite families of subgroups of a group $G$.

A $\K$-bilinear map $\eta: \K[G]^{\otimes m} \times \K[G]^{\otimes n} \ra M $ is a {\it (higher) Fox pairing (of type ($m,n$))} if the right evaluation $ \eta_{b}: G^m \ra M $ is an $m$-cocycle in $C^m_r(G,\mathcal{S};M)$ and the left evaluation $_a \eta: G^n \ra M $ is an $n$-cocycle in $C^n_l(G,\mathcal{S}';M)$ for any $ a \in \K[G]^{\otimes m} , b \in \K[G]^{\otimes n} $.

We denote the module of all such higher Fox pairings by $\mathcal{HFP}_{m,n}(G,\sh,\sh ';M).$

\end{defn}
If $m=n=1$, this definition is obviously the (original) Fox pairing. Take two copies, $G_1, G_2$, of $G$. For a cocycle $f \in Z^m_l(G_2 ; Z^n_r(G_1 ; M))$, we define a map $ \eta_{f}: (G_2)^m \times (G_1)^n \ra M $ by $ \eta_{f}(g,h):= (f(g))(h)$, which bilinearly extends to $ \K[G]^{\otimes m } \times \K[G]^{\otimes n } \ra M $ as a higher Fox pairing. In the same way as Proposition \ref{Keypro}, we can easily prove the following.
\begin{prop}\label{Keypro2}
Let $M$ be a $\K[G]$-bimodule. The correspondence $f \mapsto \eta_f$ gives rise to a $\K$-module isomorphism,
$$ Z^m_l (G_2,\sh ; Z^n_r(G_1,\sh'; M)) \cong \mathcal{HFP}_{m,n}(G,\sh,\sh ';M) . $$
\end{prop}
Furthermore, as in the proof of Theorem \ref{Thm6}, we can easily make a generalization as follows:
\begin{thm}\label{Thm6ht}
Let $(G,\sh)$ be a $\K$-$\mathrm{PD}_n$ pair with $\sh \neq \emptyset$. Then, there is an isomorphism,
$$H^{n-1}(G_2,\sh ; H^{n-1}(G_1; \K[G])) \cong H_{2}(G,\sh;\K). $$ 
\end{thm}
Furthermore, we define the Gysin maps of Fox pairings as follows. Let $ (G,\sh)$ be a $\K$-$\mathrm{PD}_n$ pair and $ (G',\sh')$ be a $\K$-$\mathrm{PD}_\ell$ pair. Take a group homomorphism $f: (G,\sh) \ra (G',\sh)$. Then, we define {\it a Gysin map} using the composite maps,
\[ f_!: H^{n-1}( G ; H^{n-1}(G; \K[G])) \cong H_1( G,\sh ; H_1(G, \sh ; \K[G]))\xrightarrow{ \ f_* \ } \]
\[ \ \ \ \ \ \ra H_1( G',\sh'; H_1(G',\sh' ; \K[G']))\cong H^{\ell -1}( G' ; H^{\ell -1}(G'; \K[G'])) .\]
Here, the first and last isomorphisms are due to the duality. Similarly, we can define the Gysin map starting from the relative cohomology:
\[ f_!: H^{n-1}( G,\sh ; H^{n-1}(G; \K[G])) \cong H_1( G ; H_1(G, \sh ; \K[G]))\xrightarrow{ \ f_* \ }\]
\begin{equation}\label{333ss} \ \ \ \ \ \ \ra H_1( G' ; H_1(G',\sh' ; \K[G']))\cong H^{\ell -1}( G' ,\sh'; H^{\ell -1}(G'; \K[G'])) .\end{equation}
Concerning the latter Gysin map, as in the functorial discussion in the proof of Theorem \ref{Thm6}, we can show a functorial result for Theorem \ref{Thm6ht}.
\begin{prop}\label{ex313} The Gysin map $f_!$ in \eqref{333ss} is equal to the canonical pushforward,
$$ f_* : H_2( G, \sh ;\K) \lra H_2(G',\sh' ;\K).$$
\end{prop}
\begin{exa}\label{ex353} 
If $n=\ell=2$, then the homology $H_2( G, \sh ;\K) $ is isomorphic to $H^0(G;\K) \cong \K $. Therefore, the Gysin map $f_!$ in \eqref{333ss} is always an isomorphism. Namely, for any fundamental Fox pairing $\eta$ of $(G,\sh)$, $f_!(\eta)$ is also a fundamental Fox pairing of $(G',\sh')$.
\end{exa}

\begin{exa}\label{ex339}
As in Section \ref{kkkk}, let us consider a link $L: \sqcup_{k=1}^q S^1 \ra \Sigma$ with $q$ components, where $\Sigma$ is an integral homology 3-sphere and $q \in \N$. In many cases (e.g., when the complement $\Sigma \setminus L$ is irreducible), the complement $\Sigma \setminus L$ is aspherical, and the inclusion $\iota: \partial ( \Sigma \setminus L) \hookrightarrow \Sigma \setminus L$ induces an injection $\iota_*: \pi_1( \partial ( \Sigma \setminus L)) \ra \pi_1( \Sigma \setminus L)$. If so, $ \Sigma \setminus L$ is aspherical; hence, the pair $(G,\sh)=(\pi_1( \Sigma \setminus L), \pi_1( \partial ( \Sigma \setminus L)) )$ is a $\K$-$\mathrm{PD}_3$ pair for any ring $\K$. By Alexander duality, $ H_2(G,\sh;\K) \cong \K^{q}$. Hence, Theorem \ref{Thm6ht} implies that there are $q$ non-trivial higher Fox pairings from $(G,\sh) $. An open problem is to find a way to explicitly express Fox pairings as representatives of the homology classes of $ H_2(G,\sh;\K)\cong \K^q$.

Let $F \subset \Sigma \setminus L$ be a Seifert surface whose boundary is the $k$-th component of $L$. Consider the inclusion $\iota: (\pi_1(F) , \pi_1(\partial F)) \ra (G,\sh)$, and let $ \eta_F$ be the fundamental Fox pairing of $\pi_1(F)$. Then, the fundamental Fox pairing of $(G,\sh)$ as the $k$-th basis of $ H_2(G,\sh;\K) \cong \K^{q}$ is equal to $\iota_!(\eta_F)$ by definitions.
\end{exa}

Finally, let us briefly discuss the case $\sh = \emptyset$ and higher Fox pairings for aspherical {\it closed} $n$-manifolds. Here, we will prove an easy proposition:
\begin{prop}\label{ex77733}
Let the pair $(G,\sh)$ be a $\K$-$\mathrm{PD}_n$ pair with $\sh = \emptyset$.

Then, $H^s(G;\K[G])$ is zero if $s \neq n$, and $\K$ otherwise. Furthermore, $ H^t( G; H^n(G;\K[G]))$ is isomorphic to the ordinary cohomology $H^t (G;\K)$. In particular, if $t=n$, then the $n$-th cohomology $ H^n( G; H^n(G;\K[G]))$ is isomorphic to $\K$.
\end{prop}
\begin{proof}
By duality and the Shapiro lemma, $H^s(G;\K[G]) \cong H_{n-s} (G;\K[G]) \cong H_{n-s}(\mathrm{pt};\K)$, which proves the former claim. Since $G$ acts trivially on $H_{n-k}(\mathrm{pt};\K)$, we have $ H^t( G; H^n(G;\K[G]) \cong H^t( G; \K )$, as claimed.
\end{proof}

As an example, suppose that $X$ is an aspherical closed $n$-manifold with orientation and $n\geq 2$, and $G=\pi_1(X)$; then, $ H^{n }( G; H^{n } (G;\K[G]) ) \cong \K$. This means that such a manifold admits no original Fox pairing, but does admit a non-trivial higher Fox pairing. An open problem is to find a way to concretely express higher Fox pairings as a representative $n$-cocycle. We conclude this paper by giving a higher Fox pairing of $\Z^n$. 
\begin{exa}\label{ex7733}
Let $G=\Z^n$ and $\K$ be a PID. Then, $G$ is a $\K$-$\mathrm{PD}_n$ pair, since $K(G,1)$ is a torus $(S^1)^n$. We fix a basis $t_1, \dots, t_n \in G$, 
and regard $\K[G]$ as the Laurent polynomial ring $\K[t_1^{\pm 1}, \dots, t_n^{\pm 1}]$. Notice that, if $n=1$, then a basis of $H^1(\Z;\K[\Z]) $ is represented by the derivation $D_i$, which takes $t_j^\ell$ to $\delta_{ i j}\sum_{k = 0}^{\ell -1 } t_i^{k} $ if $\ell \geq 0$, and to $\delta_{ i j}\sum_{k = 0}^{- \ell-1 } t_i^{-k} $ if $\ell < 0$. By the K\"{u}nneth theorem, $H^n(G;\K[G]) \cong \K $ is represented by the cross product $D_1 \times D_2 \times \cdots \times D_n$. In particular, the higher Fox pairing $\eta$ of a basis of $H^{n }( G; H^{n } (G;\K[G]) ) \cong \K$ is described as $ (D_1 \times \cdots \times D_n) \otimes (D_1 \times \cdots \times D_n )$. More precisely, by the definition of the cross product, the map $\eta:\K[G]^n \times \K[G]^n \ra \K[G]$ is given by 
\[ \eta\bigl( (t_1^{e_{1,1}} \cdots t_n^{e_{1,n}}, t_1^{e_{2,1}} \cdots t_n^{e_{2,n}},\dots, t_1^{e_{n,1}} \cdots t_n^{e_{n,n}}), (t_1^{f_{1,1}} \cdots t_n^{f_{1,n}}, \dots, t_1^{f_{n,1}} \cdots t_n^{f_{n,n}})\bigr) \] 
\[= \prod_{k=1}^n \bigl( D_k( t_k^{e_{k,k}}) t_k^{-\sum_{\ell=k +1}^n e_{\ell,\ell } } \bigr) \prod_{k=1}^n \bigl( D_k( t_k^{f_{k,k}}) t_k^{-\sum_{\ell=k +1}^n f_{\ell,\ell } } \bigr) .\]

\end{exa}

\subsection{Acknowledgments}
The work was partially supported by JSPS KAKENHI, Grant Number 00646903.

\appendix

\section{Fox pairings from seven-term exact sequences}
\label{SS2}
In this appendix, we observe Fox pairings from seven-term exact sequences and give the proof of Theorem \ref{Thm688}. 

Let us assume $G$ to be a group and $M$ to be a $\K[G]$-bimodule. By Proposition \ref{Keypro}, it is reasonable to discuss the cohomology $H^1(G_1 ; H^1(G_2 ; M)) $. For this, let us consider the Lyndon-Hochschild-Serre spectral sequence associated with the extension,
\begin{equation}\label{iiiij} 0 \lra G_1 \stackrel{i}{\lra} G_2 \times G_1 \lra G_2 \lra 0,\end{equation}
where $i(g) =(1,g)$ for any $g \in G_1$. Then, as is known \cite{DHW}, the tail of the spectral sequence reduces to an exact sequence of seven terms:
\begin{multline}\label{23456}
0 \ra H^1(G_2 ;M^{G_1}) \lra H^1(G_2 \times G_1 ;M) \lra H^1(G_1 ;M)^{G_2} \lra H^2(G_2 ;M^{G_1} ) \ra \\
\ra H^2(G_2 \times G_1 ;M )_1 \stackrel{\rho }{\lra}
H^1(G_2; H^1(G_1;M)) \lra H^3(G_2;M^{G_1}) .
\end{multline}
Here, $ M^{G_1}$ is the invariant part $\{ m \in M | g \cdot m =m \ \mathrm{for \ any \ } g \in G_1\} $, and the fifth term $H^2(G_2 \times G_1 ;A )_1 $ is the kernel of the restriction of $i^*$, i.e.,
$$ H^2(G_2 \times G_1 ;M )_1= \Ker(i^*: H^2(G_2 \times G_1;M)\ra H^2(G_1; M)) .$$
In particular, if $M^{G_1}=0 $ or $ H^*(G_2;M^{G_1}) =0$, then $\rho$ is an isomorphism. For example, we can easily verify that if $G$ is of infinite order and $M=\Z[G]$, then $M^{G_1}=0 $. Therefore, it is sensible to focus on the map $\rho$ and the second cohomology $H^2(G_2 \times G_1 ;M ) $.

Let us review the construction of the map $\rho $ from \cite{DHW}. Take the maps
$$f: (G_2 \times G_1) \times (G_2 \times G_1) \lra M , \ \ \ \ k: G_2 \lra M, $$
such that $f$ lies in $Z^2(G_2 \times G_1;M)$ and $\delta^1(k)(a,b)= f((1,a),(1,b))$ for any $a,b \in G_2$. Now, let us define $\rho_f : G_2 \times G_1 \ra M$ by
\begin{equation}\label{iiii} \rho_f(a,b)= f ((a,1),(1,b))- f((1,b),(a,1)) +q(a) (1-b^{-1})+(1-a) q(b)\end{equation}
for $(a,b )\in G_2 \times G_1.$ Then, we can easily verify that $ \rho_f$ lies in $ Z^1_l(G_2; Z^1_r(G_1;M)) $. According to Sections 6 and 10.3 in \cite{DHW}, the correspondence $f \mapsto \rho_f$ is equal to the map $\rho$.

Now we will give an example and make a comparison of the cross product and Example \ref{deriex}.

\begin{exa}\label{kl1} 
Suppose that $\K$ is a PID, and $M$ is $ \K [G] \otimes_{\K} \K[G]$ and that $G$ is of infinite order. Here, $M$ is acted on by
$$ g( a \otimes b)= ga \otimes b, \ ( a \otimes b) \cdot h = a \otimes bh^{-1} \ \ \ \mathrm{for} \ \ \ a,b \in G, g\in G_2, h\in G_1. $$
Then, $H^0( G_i ;\K[G]) = \K[G]^{G_i}=0$. Thus, the K\"{u}nneth theorem in local coefficients implies that the cross product
$$ \times : H^1(G_2 ;\K[G]) \otimes H^1( G_1 ;\K[G]) \lra H^2(G_2 \times G_1 ;\K [G] \otimes_{\K} \K[G] ), $$
is an isomorphism. Furthermore, since we can easily verify that $ M^{G_1}=0$, it follows from \eqref{23456} that the composite 
\begin{equation}\label{iii}\rho \circ \times : H^1(G_2 ;\K[G]) \otimes H^1( G_1 ;\K[G]) \lra H^1(G_2 ; H^1(G_1 ; M)) \end{equation}
is an isomorphism. For two derivations $D_l \in Z^1_l(G_2 ;\K[G]) ,D_r \in Z^1_r(G_1 ;\K[G]) $, we can verify by construction that $ \rho \circ (D_l \times D_r)$ is equal to $ \rho_{D_l\otimes D_r}$ in Example \ref{deriex}. To summarize, in the above situation with $M= \K [G] \otimes_{\K} \K[G]$, every Fox pairing is a sum of some Fox pairings in Example \ref{deriex} up to coboundary.
\end{exa}
Next, for practice, let us examine the case of $M= \K[G]$. Let $\mu: \K[G] \otimes_{\K} \K[G] \ra \K[G]$ be the $\K$-bilinear map defined by $g\otimes h \mapsto gh^{-1}$. Then, we have an exact sequence of $\K[G_1 \times G_2]$-modules,
$$ 0\lra \Ker(\mu ) \lra \K[G] \otimes_{\K} \K[G] \lra \K[G] \lra 0 \ \ \ \ \mathrm{(exact)}. $$
Moreover, we have the exact sequence, 
\begin{equation}\label{ijjii} H^2( G_2 \times G_1; \K[G] \otimes_{\K} \K[G]) \stackrel{\mu_*}{\lra} H^2( G_2 \times G_1; \K[G] ) \stackrel{\delta}{\lra} H^{3}( G_2 \times G_1 ;\Ker(\mu)).\end{equation}
If $ M^{G_1}= \K[G]^{G_1}=0$, the map $\rho$ in \eqref{23456} is an isomorphism. Notice that, from \eqref{iii}, any Fox pairing from the image $\mathrm{Im}(\mu_*) $ is the Fox pairing $\eta_{D_{l}\otimes D_r}$ in Example \ref{deriex}. Thus, if the connecting homomorphism $ \delta $ is not zero, there are other examples of Fox pairings, which do not arise from Example \ref{deriex}.
\begin{proof}[Proof of Theorem \ref{Thm688}]
Let us prove the first claim. Let $(G,\sh)$ be a $\mathrm{PD}_2$-pair by assumption. By the injection in \eqref{lll5454}, we may consider only cohomology classes of the non-relative $H^1( G_2; H^1(G_1;\K[G])) $. Consider the commutative diagrams,
$${\normalsize
\xymatrix{ 0 \ar[r]&B^1( G ;Z^1(G;\K[G])) \ar[d]\ar[r] &Z^1( G ;Z^1(G;\K[G])) \ar[d] \ar[r]\ar[d] & H^1( G;Z^1(G;\K[G])) \ar[r]\ar[d]& 0\\
0 \ar[r]&B^1( G ;H^1(G;\K[G])) \ar[r] &Z^1( G ;H^1(G;\K[G])) \ar[r] & H^1( G ;H^1(G;\K[G])) \ar[r]\ar[d]^{\delta}& 0\\
& & & H^2( G ;B^1(G;\K[G])). & 
}}
$$
Here, the horizontal arrows are canonical exact sequences, and the right vertical map is derived from the $\delta$-functor associated with
$$ 0 \ra B^1(G;\K[G]) \ra Z^1(G;\K[G])\ra H^1(G;\K[G]) \ra 0.$$
The bottom-right term $H^2( G ;B^1(G;\K[G])) \cong H_0(G,\sh ;B^1(G;\K[G]))$ is zero. Therefore, any element of $ H^1( G ;H^1( G;\K [G])) $ is represented by an element of $ Z^1( G ;Z^1( G;\K [G])) $, i.e., a Fox pairing, as required.

Next, to prove the second claim, we assume that $\K$ is a PID. By duality, $ H^m( G;\Ker(\mu ))$ is zero if $m\geq 2$. Therefore, the $E_2$-term $H^p(G_2;H^q(G_1 ; \Ker(\mu ) ))$ with $p>1, q>1$ of the LHS-spectral sequence from \eqref{iiiij} is zero. In particular, $E^3_{\infty}= H^3(G_2 \times G_1 ; \Ker(\mu ))$ vanishes; the map $\mu_*$ in \eqref{ijjii} is surjective.

Since any $\mathrm{PD}_2$-pair $(G, \sh)$ is of infinite order, $\K[G]^{G_1}=0 $. By assumption, the relative cohomology $H^2(G_2 \times G_1,G_1;\K[G]) $ is isomorphic to $H^2(G_2 \times G_1;\K[G])$ because of $H^2(G;\K[G])=0$. Thus, every 2-cocycle $f$ in $Z^2(G_2 \times G_1;M) $ may satisfy that the restriction of $f $ on $\{1\} \times G_1$ is zero. Thus, it follows from \eqref{23456} that every Fox pairing $[\eta] $ in the cohomology $ H^1( G_2;H^1(G_1;\K[G]))$ is represented by a 2-cocycle in $ H^2( G_2 \times G_1; \K[G] )$, or in $H^2( G_2 \times G_1; \K[G] \otimes_{\K} \K[G]) $ by the surjectivity of $\mu_*$. Since any such 2-cocycle is a sum of cross products as shown in Example \ref{kl1}, $\eta $ is represented by a sum of Fox pairings in Example \ref{deriex}, as required.
\end{proof}

\section{On quasi-derivations and cobrackets}
\label{fff}
Here, we discuss quasi-derivations, as presented in Section 2.3 of \cite{Mas}, and prove some lemmas. A {\it quasi-derivation ruled by} a Fox pairing $\eta : \K[G] \times \K[G] \ra M$ is a $\K$-linear map $q: \K[G] \ra M$ satisfying
\begin{equation}\label{t166}
q (ab)= q(a) b +aq (b) + \eta (a,b), \ \ \ \ \mathrm{for \ any \ }a,b \in \K[G].\end{equation}
\begin{lem}\label{kl2} 
Let $\eta$ be a Fox pairing, and $A= \K[G]$. Furthermore, let $M_{\rm conj}$ be $M$ with the action of $G$ defined by setting $ a \mapsto g ag^{-1}$, where $a \in M, g\in G$.

The map $\kappa_{\eta}: G^2 \ra M $ that $(g,h)$ sends to $\eta (g,h) h^{-1} g^{-1}$ is a 2-cocycle of $G$ in the coefficients $M$. Furthermore, there is a quasi-derivation ruled by $\rho$ if and only if the cohomology class of $ \kappa_{\eta}$ vanishes in $ H^2( G ;M_{\rm conj})$.

\end{lem}
\begin{proof} 
From \eqref{t1} and \eqref{t2}, we can easily check that $ \delta^2( \kappa_{\eta})=0.$ For a map $q : A \ra M$, we define $q': A \ra M$ by $q'(a):= q(a)a^{-1}$. Since equation \eqref{t166} is equivalent to $ \delta^1(q')(a,b)= \eta (a,b)b^{-1}a^{-1} $ for any $a,b \in G $, the existence of $ q$ is equivalent to $ [\kappa_{\eta}]=0 \in H^2( G ;M_{\rm conj})$.
\end{proof}
Since $H^2( G; \K [G]_{\rm conj}) $ is isomorphic to the Hochschild cohomology $ HH^2(\K[G])$ (see, e.g., \cite[Corollary 9.7.5]{Wei}), sometimes it can be computed by using the techniques of Hochschild cohomology. Concerning Poincar\'{e} duality groups, we can immediately show the following.
\begin{lem}\label{lemkl2}
If $( G, \sh)$ is a $\mathrm{PD}_n$-pair with $n \geq 2$ and $\sh \neq \emptyset$, then $H^2( G; \K [G]_{\rm conj})$ vanishes. In particular, any Fox pairing $\eta$ of $G$ admits a quasi-derivation ruled by $\eta$.
\end{lem}
According to \cite[Lemma 2.6]{Mas}, if a Fox pairing $\eta$ satisfies $(\eta +\eta^t)(g,h)=k(1-g)(1-h )$ for some $k \in \K$ and there is a quasi-derivation ruled by $\eta$, we can define a cobracket $\K[G/\mathrm{conj}] \ra \K[G/\mathrm{conj}]\otimes \K[G/\mathrm{conj}] $. It would be interesting to construct such cobrackets starting from a Poincar\'{e} duality pair $(G,\sh)$, e.g., when $G$ arises from an orbifold as in Section \ref{SS234446}.

\normalsize

\noindent
Department of Mathematics, Tokyo Institute of Technology
2-12-1 Ookayama, Meguro-ku, Tokyo 152-8551 Japan

\end{document}